\documentclass[a4paper,10pt]{amsart}
\usepackage{amsmath,amssymb}
\usepackage{amscd}
\newtheorem{thm}{Theorem}[section]

\newtheorem{lemma}[thm]{Lemma}

\theoremstyle{remark}

\newcommand{\Ad}{{\rm{Ad}}}

\newcommand{\BN}{\mathbf N}
\newcommand{\BC}{\mathbf C}

\newcommand{\BK}{\mathbf K}

\newcommand{\la}{\langle}
\newcommand{\ra}{\rangle}

\newcommand{\Aut}{{\rm{Aut}}}

\newtheorem{Def}{Definition}[section]

\title{Fell bundles over a countable discrete group and strong Morita equivalence for
inclusions of $C^*$-algebras}
\author{Kazunori Kodaka}
\address{Department of Mathematical Sciences, Faculty of Science, Ryukyu
\endgraf
University, Nishihara-cho, Okinawa, 903-0213, Japan}
\address{\sl{E-mail address}: \rm{kodaka@math.u-ryukyu.ac.jp}}



\begin{document}
\begin{abstract}
We consider two saturated Fell bundles over a countable discrete group, whose unit fibers are $\sigma$-unital
$C^*$-algebras. Then by taking the reduced cross-sectional $C^*$-algebras, we get two inclusions of $C^*$-algebras.
We suppose that they are strongly Morita equivalent as inclusions of $C^*$-algebras.
Also, we suppose that one of the inclusions of $C^*$-algebras is irreducible, that is, the relative commutant
of one of the unit fiber algebras, which is a $\sigma$-unital $C^*$-algebra, in the multiplier $C^*$-algebra of the reduced
cross-sectional $C^*$-algebra is trivial. We show that the two saturated Fell bundles are then equivalent up to some
automorphism of the group.
\end{abstract}
 
 \maketitle
 
\section{Introduction}\label{sec:intro} Let $G$ be a countable discrete group and let $\mathcal{A}=\{A_t\}_{t\in G}$
be a Fell bundle over $G$. Let $C_r^* (\mathcal{A})$ be the reduced cross-sectional $C^*$-algebra of $\mathcal{A}$
and $A_e =A$, a $C^*$-algebra, where $e$ is the unit element in $G$. Then we obtain an inclusion of $C^*$-algebras
$A\subset C_r^* (\mathcal{A})$ and we call it the inclusion of $C^*$-algebras
\sl
induced
\rm
by $\mathcal{A}$.
By Abadie and Ferraro \cite [Sections 3, 4]{AF:Fell}, it is easy to show that if
Fell bundles $\mathcal{A}=\{A_t \}_{t\in G}$ and $\mathcal{B}=\{B_t \}_{t\in G}$ over
$G$ are equivalent with respect to an equivalence bundle $\mathcal{X}=\{X_t \}_{t\in G}$ over $G$
such that
$$
\overline{{}_{\mathcal{A}} \la X_t \, , \, X_s \ra}=A_{ts^{-1}} \quad , \quad
\overline{\la X_t \, , \, X_s \ra_{\mathcal{B}}}=B_{t^{-1}s}
$$
for any $t, s\in G$, then the inclusions of $C^*$-algebras induced by $\mathcal{A}$ and $\mathcal{B}$
are strongly Morita equivalent, where $\overline{{}_{\mathcal{A}} \la X_t \, , \, X_s \ra}$ means
the closure of linear span of the set
$$
\{{}_{\mathcal{A}} \la x, y \ra \, | \, x\in X_t \, , \, y\in X_s \}
$$
and $\overline{\la X_t \, , \, X_s \ra_{\mathcal{B}}}$ means the closure of the same set as above.
\par
In this paper, we shall show the inverse direction as follows: Let $\mathcal{A}=\{A_t\}_{t\in G}$
and $\mathcal{B}=\{B_t\}_{t\in G}$ be saturated Fell bundls over $G$. we suppose that $A_e =A$ and $B_e =B$
are $\sigma$-unital $C^*$-algebras and that $A' \cap C_r^* (\mathcal{A})=\BC1$.
If the inclusions of $C^*$-algebras induced
by $\mathcal{A}$ and $\mathcal{B}$ are strongly Morita equivalent, then there is an automorphism $f$
of $G$ such that $\mathcal{A}$ and $\mathcal{B}^f$ are equivalent with respect to an equivalence bundle
$\mathcal{X}=\{X_t\}_{t\in G}$ such that
$$
\overline{{}_{\mathcal{A}} \la X_t \, , \, X_s \ra}=A_{ts^{-1}} \quad , \quad
\overline{\la X_t \, , \, X_s \ra_{\mathcal{B}}}=B_{t^{-1}s}
$$
for any $t, s\in G$, where $\mathcal{B}^f$ is a Fell bundle over $G$ induced by $\mathcal{B}=\{B_t\}_{t\in G}$
and $f$, that is, $\mathcal{B}^f =\{B_{f(t)}\}_{t\in G}$.
\par
We prove this result in the following way: Let $\BK$ be the $C^*$-algebra of all compact operators on a countably
infinite dimensional Hilbert space. Let $\mathcal{A}^S$ and $\mathcal{B}^S$ be the Fell bundles over $G$
induced by $\mathcal{A}$, $\BK$ and $\mathcal{B}$, $\BK$, respectively. Then since $A$ and $B$ are
$\sigma$-unital, $\mathcal{A}^S$ and $\mathcal{B}^S$ satisfy the assumptions of Exel
\cite [Theorem 7.3]{Exel:regular}. Since $\mathcal{A}^S$ and $\mathcal{B}^S$ are saturated,
by \cite [Theorem 7.3]{Exel:regular} there are twisted actions $(\alpha, w_{\alpha})$ and $(\beta, w_{\beta})$
of $G$ on the $C^*$-algebras $A\otimes \BK$ and $B\otimes\BK$, which are the unit fibre algebras of
$\mathcal{A}^S$ and $\mathcal{B}^S$, such that $\mathcal{A}^S$ and $\mathcal{B}^S$ are isomorphic to the semidirect
product bundles of $A\otimes\BK$, $G$ and $B\otimes\BK$, $G$ constructed by $(\alpha, w_{\alpha})$
and $(\beta, w_{\beta})$ as Fell bundles over $G$, respectively.
Since the inclusions of $C^*$-algebras induced by $\mathcal{A}$ and $\mathcal{B}$ are strongly Morita equivalent,
so are the inclusions of $C^*$-algebras induced by $\mathcal{A}^S$ and $\mathcal{B}^S$.
Hence the inclusions of $C^*$-algebras induced by
$(\alpha, w_{\alpha})$ and $(\beta, w_{\beta})$ are strongly Morita equivalent. Then since the inclusions are
irreducible, by \cite [Theorem 5.5] {Kodaka:discrete} there is an automorphism $f$ of $G$ such that
$(\alpha, w_{\alpha})$ and $(\beta^f , w_{\beta}^f )$ are strongly Morita equivalent, where $(\beta^f , w_{\beta}^f )$ is
the twisted action of $G$ on $B\otimes\BK$ induced by $(\beta, w_{\beta})$ and $f$, that is,
$$
\beta_t^f =\beta_{f(t)} \quad , \quad w_{\beta}^f (t, s)=w_{\beta}(f(t)\, , \, f(s))
$$
for any $t, s\in G$. Using this, we can prove the result.

\section{Preliminaries}\label{sec:pre} Let $A$ be a $C^*$-algebra and we denote
by $M(A)$ the multiplier $C^*$-algebra of $A$. Let $\alpha$ be an automorphism of $A$.
Then there is a unique strictly continuous automorphism of $M(A)$ extending $\alpha$ by Jensen and
Thomsen \cite [Corollary 1.1.15]{JT:KK}. We denote it by $\underline{\alpha}$.
\par
Let $G$ be a countable discrete group and $e$ the unit element
in $G$. Let $\mathcal{A}=\{A_t\}_{t\in G}$ be a Fell bundle over $G$ and let $A_e =A$, a $C^*$-algebra.
Also, let $\mathcal{B}=\{B_t\}_{t\in G}$ be a Fell bundle over $G$ and let $B_e =B$, a $C^*$-algebra.
Following Abadie and Ferraro \cite [Definitions 2.1 and 2.2]{AF:Fell}, we give the definition of an equivalence bundle:

\begin{Def}\label{def:pre0} (1) A
\sl
right Hilbert $\mathcal{B}$-bundle
\rm
is a complex Banach bundle over $G$, $\mathcal{X}=\{X_t \}_{t\in G}$ with continuous maps
$$
\mathcal{X}\times\mathcal{B}\to\mathcal{X}, \, (x, b)\mapsto xb \quad \text{and} \quad \la - , - \ra_{\mathcal{B}}:
\mathcal{X}\times\mathcal{X}\to\mathcal{B}, \, (x, y)\mapsto \la x, y \ra_{\mathcal{B}}
$$
such that:
\newline
(1R) $X_r B_s \subset X_{rs}$ and $\la X_r \, , \, X_s \ra_{\mathcal{B}}\subset B_{r^{-1}s}$ for all $r, s\in G$.
\newline
(2R) $X_r \times B_s \to X_{rs}$, $(x, b)\mapsto xb$ is bilinear for all $r, s\in G$.
\newline
(3R) $X_s \to B_{r^{-1}s}$, $y\mapsto \la x, y \ra_{\mathcal{B}}$ is linear for all $x\in X_r$ and $s\in G$.
\newline
(4R) $\la x, yb \ra_{\mathcal{B}}=\la x, y \ra_{\mathcal{B}}b$ and $\la x, y \ra_{\mathcal{B}}^*= \la y, x \ra_{\mathcal{B}}$
for all $x, y\in \mathcal{X}$ and $b\in \mathcal{B}$.
\newline
(5R) $\la x, x \ra_{\mathcal{B}}\geq 0$ for all $x\in \mathcal{X}$ and $\la x , x \ra_{\mathcal{B}}=0$ implies $x=0$.
Beside, each fiber $X_t$ is complete with respect to the norm $x\mapsto ||\la x, x \ra_{\mathcal{B}}||^{1/2}$.
\newline
(6R) For all $x\in \mathcal{X}$, $||x||^2 =||\la x, x \ra_{\mathcal{B}}||$.
\newline
(7R) $\overline{\{\la X_s \, , \, X_s \ra_{\mathcal{B}} \, | \, s\in G \}}=B_e$.
\par
(2) A
\sl
left Hilbert $\mathcal{A}$-bundle
\rm
is a complex Banach bundle over $G$, $\mathcal{X}=\{X_t \}_{t\in G}$ with continuous maps
$$
\mathcal{A}\times\mathcal{X}\to\mathcal{X}, \, (a, x)\mapsto ax \quad \text{and} \quad {}_{\mathcal{A}}\la - , - \ra:
\mathcal{X}\times\mathcal{X}\to\mathcal{A}, \, (x, y)\mapsto {}_{\mathcal{A}}\la x, y \ra
$$
such that:
\newline
(1L) $A_r X_s \subset X_{rs}$ and ${}_{\mathcal{A}}\la X_r \, , \, X_s \ra\subset A_{rs^{-1}}$ for all $r, s\in G$.
\newline
(2L) $A_r \times X_s \to X_{rs}$, $(a, x)\mapsto ax$ is bilinear for all $r, s\in G$.
\newline
(3L) $X_s \to A_{sr^{-1}}$, $y\mapsto {}_{\mathcal{A}}\la y, x \ra$ is linear for all $x\in X_r$ and $s\in G$.
\newline
(4L) ${}_{\mathcal{A}}\la ax, y \ra=a{}_{\mathcal{A}}\la x, y \ra$ and ${}_{\mathcal{A}}\la x, y \ra^*=
{}_{\mathcal{A}}\la y, x \ra$ for all $x, y\in \mathcal{X}$ and $a\in \mathcal{A}$.
\newline
(5L) ${}_{\mathcal{A}}\la x, x \ra\geq 0$ for all $x\in \mathcal{X}$ and ${}_{\mathcal{A}}\la x , x \ra=0$ implies $x=0$.
Beside, each fiber $X_t$ is complete with respect to the norm $x\mapsto ||{}_{\mathcal{A}}\la x, x \ra||^{1/2}$.
\newline
(6L) For all $x\in \mathcal{X}$, $||x||^2 =||{}_{\mathcal{A}}\la x, x \ra||$.
\newline
(7L) $\overline{\{{}_{\mathcal{A}}\la X_s \, , \, X_s \ra \, | \, s\in G \}}=A_e$.
\par
(3) We say that $\mathcal{X}$ is an $\mathcal{A}-\mathcal{B}$-
\sl
equivalence bundle
\rm
if $\mathcal{X}$ is both a left Hilbert $\mathcal{A}$-bundle, a right Hilbert $\mathcal{B}$-bundle and
${}_{\mathcal{A}} \la x, y \ra z=x \la y, z \ra_{\mathcal{B}}$ for all $x, y, z\in \mathcal{X}$. Besides,
we say $\mathcal{A}$ is equivalent to $\mathcal{B}$ if there exists an $\mathcal{A}-\mathcal{B}$-equivalence bundle.
\end{Def}

\begin{Def}\label{def:pre1} Let $\mathcal{A}=\{A_t\}_{t\in G}$ be a Fell bundle over $G$. We say that
$\mathcal{A}=\{A_t\}_{t\in G}$ is
\sl
saturated
\rm
if $\overline{A_t A_t^*}=A_e$ for any $t\in G$.
\end{Def}

Let $\mathcal{A}=\{A_t \}_{t\in G}$ be a saturated Fell bundle over $G$ and
let $A_e =A$, a $C^*$-algebra. We suppose that
$A$ is a $\sigma$-unital $C^*$-algebra.
\par
Let $C_r^* (\mathcal{A})$ be the reduced cross-sectional $C^*$-algebra of $\mathcal{A}$.
Then for any $t\in G$, $A_t$ is regarded as a closed subspace of $C_r^* (\mathcal{A})$ since
$G$ is discrete.
\par
Let $\BK$ be the $C^*$-algebra of all compact operators on a countably infinite dimensional Hilbert
space. Let $\{e_{ij}\}_{i, j\in\BN}$ be a system of matrix units of $\BK$.
\par
Let $A_t \otimes\BK$ be the closure of linear span of the subset
$$
\{x\otimes k\in C_r^* (\mathcal{A})\otimes\BK \, | \, x\in A_t \, , \, k\in\BK \} .
$$
Let $\mathcal{A}^S=\{A_t \otimes\BK \}_{t\in G}$. Then $\mathcal{A}^S$ is a saturated Fell bundle over $G$
and $A\otimes\BK$ is its unit fibre algebra. Clearly $C_r^* (\mathcal{A}^S )=C_r^* (\mathcal{A})\otimes\BK$.
Since $\mathcal{A}$ is saturated, we can regard $A_t$ as an $A-A$-equivalence bimodule for any
$t\in G$ and by the definition of the product in $C_r^* (\mathcal{A}^S )$ we can regard $A_t \otimes\BK$ as
the tensor product of the $A-A$-equivalence bimodule $A_t$ and the trivial $\BK-\BK$-equivalence
bimodule $\BK$, which is $A\otimes\BK-A\otimes\BK$-equivalence bimodule for any $t\in G$.
Thus $\mathcal{A}^S$ is a saturated Fell bundle over $G$. Since $A\otimes\BK$ is $\sigma$-unital,
by \cite [Theorem 7.3]{Exel:regular}, there is a twisted action $(\alpha, w_{\alpha})$ of $G$ on
$A\otimes\BK$ such that $\mathcal{A}^S$ is isomorphic to the semidirect product bundle over $G$
induced by $(\alpha, w_{\alpha})$ as Fell bundles over $G$.

\begin{lemma}\label{lem:pre2} With the above notation, if $A' \cap M(C_r^* (\mathcal{A}))=\BC1$, then
$(A\otimes\BK)' \cap M(A\rtimes_{\alpha, w_{\alpha}, r} G)=\BC 1$.
\end{lemma}
\begin{proof} Since $A' \cap M(C_r^* (\mathcal{A}))=\BC 1$, by \cite [Lemma 3.1]{Kodaka:discrete}
$(A\otimes\BK)' \cap M(C_r^* (\mathcal{A})\otimes\BK)=\BC 1$. Since $\mathcal{A}^S$ is isomorphic to the
semidirect product bundle over $G$ of $A\otimes\BK$ induced by $(\alpha, w_{\alpha})$ as
Fell bundles over $G$, the inclusions $A\otimes\BK\subset C_r^* (\mathcal{A})\otimes\BK$ and
$A\otimes\BK\subset (A\otimes\BK)\rtimes_{\alpha, w_{\alpha}, r}G$ are isomorphic as inclusions of
$C^*$-algebras. Thus $(A\otimes\BK)' \cap M((A\otimes\BK)\rtimes_{\alpha, w_{\alpha}, r}G)=\BC1$.
\end{proof}

Let $(\alpha, w_{\alpha})$ and $(\beta, w_{\beta})$ be twisted actions of a countable discrete
group $G$ on $C^*$-algebras $A$ and $B$, respectively. We suppose that $(\alpha, w_{\alpha})$ and
$(\beta, w_{\beta})$ are strongly Morita equivalent with respect to a twisted action $\lambda$ of
$G$ on an $A-B$-equivalence bimodule $X$, that is, $\lambda$ is a map from $G$ to $\Aut(X)$
satisfying the following:
\newline
(1) $\alpha_t ({}_A \la x, y \ra)={}_A \la \lambda_t (x)\, , \, \lambda_t (y) \ra$,
\newline
(2) $\beta_t (\la x, y \ra_B )=\la \lambda_t (x) \, , \, \lambda_t (y) \ra_B$,
\newline
(3) $(\lambda_t \circ \lambda_s )(x)=w_{\alpha}(t, s)\lambda_{ts}(x)w_{\beta}(t, s)^*$
\newline
for any $t, s\in G$, $x, y\in X$, where we regard $X$ as a Hilbert $M(A)-M(B)$-bimodule as in
\cite [Prelominaries]{Kodaka:discrete}.
\par
Let $u$ and $v$ be unitary representations of $G$ to $M(A\rtimes_{\alpha, w_{\alpha}, r}G)$ and
$M(B\rtimes_{\beta, w_{\beta}, r}G)$ implementing $\alpha$ and $\beta$, respectively, that is,
$\alpha_t =\Ad(u_t )$ and $\beta_t =\Ad(v_t )$ for any $t\in G$.
Let $\mathcal{A}=\{Au_t \}_{t\in G}$ and $\mathcal{B}=\{Bv_t \}_{t\in G}$ be the semidirect
product Fell bundles over $G$ induced by $(\alpha, w_{\alpha})$ and $(\beta, w_{\beta})$, respectively.
\par
For any $t\in G$, let $X_t =Xv_t$ as Banach spaces. We regard $\mathcal{X}=\{X_t\}_{t\in G}$ as
an $\mathcal{A}-\mathcal{B}$-equivalence bundle in the following (See \cite {Exel:regular} and
\cite {KW2:discrete}): For any $au_t\in Au_t$, $bv_t \in Bv_t$,
$xv_s \in Xv_s$, $yv_t \in Xv_t$,
\begin{align*}
(au_t )(xv_s ) & =a\lambda_t (x)w_{\beta}(t, s)v_{ts} , \\
(xv_s )(bv_t ) & =x\beta_s (b)w_{\beta}(s, t)v_{st} , \\
{}_{\mathcal{A}} \la xv_s \, , \, yv_t \ra & ={}_A \la x \, , \, (\lambda_s \circ\lambda_t^{-1})(y) \ra
w_{\alpha}(t, s^{-1})\underline{\alpha_{ts^{-1}}}(w_{\alpha}(s, s^{-1}))^* u_{ts^{-1}} , \\
\la xv_s \, , \, yv_t \ra_{\mathcal{B}} & =\beta_s^{-1}(\la x, y \ra_B )
w_{\beta}(s^{-1} \, , \, s)^* w_{\beta}(s^{-1} \, , \, t)v_{s^{-1}t} .
\end{align*}

\begin{lemma}\label{lem:pre3} With the above notation, $\mathcal{X}=\{X_t \}_{t\in G}$ is an
$\mathcal{A}-\mathcal{B}$-equivalence bundle over $G$ such that
$$
\overline{{}_{\mathcal{A}} \la X_t \, , \, X_s \ra}=A_{ts^{-1}} \, , \, \overline{\la X_t \, , \, X_s \ra_{\mathcal{B}}}
=B_{t^{-1}s}
$$
for any $t, s\in G$.
\end{lemma}
\begin{proof} By the definition of $\mathcal{X}=\{X_t\}_{t\in G}$, it is clear that $\mathcal{X}$ has
Conditions (1R)-(4R) and (1L)-(4L) in Definition \ref{def:pre0}. For any $xv_t \in Xv_t$,
\begin{align*}
||\la xv_t \, , \, xv_t \ra_{\mathcal{B}}|| & =||\beta_t^{-1}(\la x, x \ra_B )||
=||\la x, x \ra_B ||=||x||^2 =||xv_t ||^2 , \\
||{}_{\mathcal{A}} \la xv_t \, , \, xv_t \ra || & =||{}_A \la x, x \ra||=||xv_t ||^2 .
\end{align*}
Hence we see that $\{X_t \}_{t\in G}$ has Conditions (5R), (6R) and (5L), (6L) in Definition \ref{def:pre0}.
Furthermore, since
$X$ is an $A-B$-equivalence bimodule, we can see that $\mathcal{X}$ satisfies that
$$
\overline{{}_{\mathcal{A}} \la X_t \, , \, X_s \ra}=A_{ts^{-1}} \, , \, \overline{\la X_t \, , \, X_s \ra_{\mathcal{B}}}
=B_{t^{-1}s}
$$
for any $t, s\in G$.
\end{proof}

\section{Strong Morita equivalence}\label{sec:SM} Let $\mathcal{A}=\{A_t \}_{t\in G}$ and
$\mathcal{B}=\{B_t \}_{t\in G}$ be saturated Fell bundles over $G$. Let $A_e =A$ and $B_e =B$
be $C^*$-algebras.
Let $\mathcal{A}^S$ and $\mathcal{B}^S$ be the saturated Fell bundles over $G$ induced by
$\mathcal{A}$, $\BK$ and $\mathcal{B}$, $\BK$, respectively. Let $(\alpha, w_{\alpha})$ and
$(\beta, w_{\beta})$ be the twisted actions of $G$ on $A\otimes\BK$ and $B\otimes\BK$ such that
$\mathcal{A}^S$ and $\mathcal{B}^S$ are isomorphic to the semidirect product bundles of
$A\otimes\BK$ and $B\otimes\BK$ induced by $(\alpha, w_{\alpha})$ and $(\beta, w_{\beta})$,
which are defined in Section \ref{sec:pre}, respectively. We suppose that $A\subset C_r^* (\mathcal{A})$
and $B\subset C_r^* (\mathcal{B})$ are strongly Morita equivalent and that
$A' \cap M(C_r^* (\mathcal{A}))=\BC 1$.

\begin{lemma}\label{lem:SM1} With the above notation and assumptions, there is an automorphism $f$
of $G$ such that $(\alpha, w_{\alpha})$ is strongly Morita equivalent to $(\beta, w_{\beta}^f )$,
$(\beta^f , w_{\beta}^f )$ is the twisted action of $G$ on $B\otimes\BK$ induced by $(\beta, w_{\beta})$
and $f$, which is defined by
$$
\beta_t^f =\beta_{f(t)} \, , \, w_{\beta}^f (t, s)=w_{\beta}(f(t), f(s))
$$
for any $t, s\in G$.
\end{lemma}
\begin{proof} Since $A' \cap M(C_r^* (\mathcal{A}))=\BC 1$, by Lemma \ref {lem:pre2}
$$
(A\otimes\BK)' \cap M((A\otimes\BK)\rtimes_{\alpha, w_{\alpha}, r}G)=\BC 1 .
$$
Hence by \cite [Theorem 5.5]{Kodaka:discrete} we obtain the conclusion.
\end{proof}

Let $\mathcal{B}^{S, f}=\{B_{f(t)}\otimes\BK\}_{t\in G}$ be the semidirect product bundle of $B\otimes\BK$
induced by $(\beta^f , w_{\beta}^f )$. Then by Lemma \ref{lem:pre3} and Lemma \ref {lem:SM1},
there is an $\mathcal{A}^S -\mathcal{B}^{S, f}$-equivalence bundle $\mathcal{Y}=\{Y_t \}_{t\in G}$
such that
$$
\overline{{}_{\mathcal{A}^S} \la Y_t \, , \, Y_s \ra}=A_{ts^{-1}}\otimes\BK \quad , \quad
\overline{\la Y_t \, , \, Y_s \ra_{\mathcal{B}^{S, f}}}=B_{f(t^{-1}s)}\otimes\BK
$$
for any $t, s\in G$.
\par
Let $\mathcal{B}^f=\{B_{f(t)}\}_{t\in G}$ be the Fell bundle over $G$ induced by $\mathcal{B}=\{B_t \}_{t\in G}$
and the automorphism $f$ of $G$.

\begin{lemma}\label{lem:SM2} With the above notation, $\mathcal{A}$ and $\mathcal{B}^f$
are equivalent with respect to an $\mathcal{A}-\mathcal{B}^f$-equivalence bundle
$\mathcal{X}=\{X_t \}_{t\in G}$ such that
$$
\overline{{}_{\mathcal{A}} \la X_t \, , \, X_s \ra}=A_{ts^{-1}} \quad , \quad
\overline{\la X_t \, , \, X_s \ra_{\mathcal{B}^f}}=B_{f(t^{-1}s)}
$$
\end{lemma}
\begin{proof} Let $X_t =(1\otimes e_{11})Y_t (1\otimes e_{11})$ for any $t\in G$.
Let $t, s\in G$ and let $a_t \in A_t$, $b_t \in B_t$, $x_s \in Y_s$ and $x_t , y_t \in Y_t$.
Then
\begin{align*}
a_t(1\otimes e_{11})x_s (1\otimes e_{11}) & =(1\otimes e_{11})(a_t \otimes e_{11})x_s (1\otimes e_{11})\in
X_{ts} , \\
(1\otimes e_{11})x_s (1\otimes e_{11})b_t & =(1\otimes e_{11})x_s (b_t \otimes e_{11})(1\otimes e_{11})\in
X_{st} ,
\end{align*}
\begin{align*}
& {}_{\mathcal{A}^S} \la (1\otimes e_{11})x_s (1\otimes e_{11}) \, , \, (1\otimes e_{11})y_t (1\otimes e_{11}) \ra \\
& =(1\otimes e_{11}) {}_{\mathcal{A}^S} \la x_s (1\otimes e_{11}) \, ,\,
y_t (1\otimes e_{11} \ra(1\otimes e_{11})
\in (1\otimes e_{11})(A_{st^{-1}}\otimes\BK)(1\otimes e_{11}),
\end{align*}
where we identify $(1\otimes e_{11})(A_{st^{-1}}\otimes\BK)(1\otimes e_{11})$
with $A_{st^{-1}}$.
Also,
\begin{align*}
& \la (1\otimes e_{11})x_s (1\otimes e_{11}) \, , \, (1\otimes e_{11})y_t (1\otimes e_{11}) \ra_{\mathcal{B}^{S, f}} \\
& =(1\otimes e_{11}) \la (1\otimes e_{11})x_s \, , \, (1\otimes e_{11})y_t \ra_{\mathcal{B}^{S, f}}(1\otimes e_{11})
\in (1\otimes e_{11})(B_{f(s^{-1}t)}\otimes\BK)(1\otimes e_{11}) ,
\end{align*}
where we identify $(1\otimes e_{11})(B_{f(s^{-1}t)}\otimes\BK)(1\otimes e_{11})$ with $B_{f(s^{-1}t)}$.
Hence $\mathcal{X}$ has Conditions (1R)-(4R) and (1L)-(4L) in Definition \ref {def:pre0}.
Furthermore,
\begin{align*}
& {}_{\mathcal{A}^S} \la (1\otimes e_{11})x_t (1\otimes e_{11}) \, , \, (1\otimes e_{11})x_t (1\otimes e_{11}) \ra \\
& =(1\otimes e_{11}) {}_{\mathcal{A}^S} \la x_t (1\otimes e_{11}) \, , \, x_t (1\otimes e_{11}) \ra
(1\otimes e_{11}) , \\
& \la (1\otimes e_{11})x_t (1\otimes e_{11}) \, , \, (1\otimes e_{11})x_t (1\otimes e_{11})
\ra_{\mathcal{B}^{S, f} }\\
& =(1\otimes e_{11}) \la (1\otimes e_{11})x_t \, , \, (1\otimes e_{11})x_t \ra_{\mathcal{B}^{S, f}}
(1\otimes e_{11}) .
\end{align*}
These equations implies that $\mathcal{X}$ has Conditions (5R) and (5L) in Definition \ref{def:pre0}.
It is clear that $\mathcal{X}$ has Conditions (6R) and (6L) in Definition \ref {def:pre0}.
Moreover, let $c, d \in B\otimes\BK$. Then
\begin{align*}
& {}_{\mathcal{A}^S} \la (1\otimes e_{11})x_s c(1\otimes e_{11})\, , \,
(1\otimes e_{11})y_t d (1\otimes e_{11}) \ra \\
& =(1\otimes e_{11}){}_{\mathcal{A}^S} \la x_s c(1\otimes e_{11})d^* \, , \, y_t \ra (1\otimes e_{11}) .
\end{align*}
Since $1\otimes e_{11}$ is full in $B\otimes \BK$, that is,
$\overline{(B\otimes\BK)(1\otimes e_{11})(B\otimes\BK)}=B\otimes\BK$ and
$\overline{Y_t (B\otimes\BK)}=Y_t$ by \cite [Proposition 1.7(i)]{BMS:quasi},
$$
\overline{{}_{\mathcal{A}^S} \la (1\otimes e_{11}) x_s (1\otimes e_{11})\, , \, (1\otimes e_{11})y_t (1\otimes e_{11}) \ra}
=A_{st^{-1}}\otimes e_{11} .
$$
In the same way, we can see that
$$
\overline{\la (1\otimes e_{11})x_s (1\otimes e_{11}) \, ,\, (1\otimes e_{11})y_t (1\otimes e_{11})
\ra_{\mathcal{B}^{S, f}}}=B_{f(s^{-1}t)}\otimes e_{11} .
$$
Hence 
$$
\overline{{}_{\mathcal{A}} \la X_t \, , \, X_s \ra}=A_{ts^{-1}} \quad , \quad
\overline{\la X_t \, , \, X_s \ra_{\mathcal{B}^f}}=B_{f(t^{-1}s)}
$$
for any $t, s\in G$. Since for any $x, y, z\in \mathcal{Y}$,
${}_{\mathcal{A}^S} \la x, y \ra z=x \la y, z \ra_{\mathcal{B}^{S, f}}$, we can see that 
\begin{align*}
& {}_{\mathcal{A}} \la (1\otimes e_{11})x(1\otimes e_{11}) \, , \, (1\otimes e_{11})y(1\otimes e_{11}) \ra
(1\otimes e_{11})z(1\otimes e_{11}) \\
& =(1\otimes e_{11})x(1\otimes e_{11}) \la (1\otimes e_{11})y (1\otimes e_{11}) \, , \,
(1\otimes e_{11})z(1\otimes e_{11}) \ra_{\mathcal{B}^f}
\end{align*}
for any $x, y, z\in \mathcal{Y}$. It follows that $\mathcal{X}$ is an $\mathcal{A}-\mathcal{B}^f$-
equivalence bundle. Therefore, we obtain the conclusion.
\end{proof}

\begin{thm}\label{thm:SM3} Let $G$ be a countable discrete group and let $\mathcal{A}=\{A_t\}_{t\in G}$
and $\mathcal{B}=\{B_t \}_{t\in G}$ be saturated Fell bundles over $G$. We suppose that
$A_e =A$ and $B_e =B$ are $\sigma$-unital $C^*$-algebras and that $A' \cap C_r^* (\mathcal{A})=\BC 1$,
where $e$ is the unit element in $G$ and $C_r^* (\mathcal{A})$ is the reduced cross-sectional
$C^*$-algebra of $\mathcal{A}$. If the inclusions of $C^*$-algebras induced by $\mathcal{A}$ and
$\mathcal{B}$ are strongly Morita equivalent, then there is an automorphism $f$ of $G$ such that
$\mathcal{A}$ and $\mathcal{B}^f$ are equivalent with respect to an $\mathcal{A}-\mathcal{B}^f$-
equivalence bundle $\mathcal{X}=\{X_t \}_{t\in G}$ over $G$ such that
$$
\overline{{}_{\mathcal{A}} \la X_t \, , \, X_s \ra}=A_{ts^{-1}} \quad , \quad
\overline{\la X_t \, , \, X_s \ra_{\mathcal{B}^f}}=B_{f(t^{-1}s)}
$$
for any $t, s\in G$, where $\mathcal{B}^f$ is a Fell bundle over $G$ induced by
$\mathcal{B}=\{B_t \}_{t\in G}$ and $f$, that is, $\mathcal{B}^f =\{B_{f(t)}\}_{t\in G}$.
\end{thm}
\begin{proof} This is immediate by Lemmas \ref{lem:SM1}, \ref{lem:SM2} and the discussions before
Lemma \ref{lem:SM2}.
\end{proof}

\end{document}